\documentclass{amsart}
\usepackage{amsfonts,amssymb,amsmath,amsthm}
\usepackage{url}
\usepackage{enumerate}

\usepackage[mathscr]{eucal}
\usepackage{amsthm,amsmath,amssymb,amscd}
\usepackage{graphicx}
\usepackage{caption}
\usepackage{subcaption}
\usepackage{latexsym}
\usepackage{color}
\newtheorem{theorem}{Theorem}[section]

\newtheorem{lemma}[theorem]{Lemma}
\newtheorem{proposition}[theorem]{Proposition}
\newtheorem{corollary}[theorem]{Corollary}
\theoremstyle{definition}
\newtheorem{definition}[theorem]{Definition}
\newtheorem{remark}[theorem]{Remark}
\newtheorem*{example}{Example}

\numberwithin{equation}{section}
\def \beq{ \begin{equation} }
\def \eeq{\end{equation}}
\def \pd{\partial}

\def \R{\mathbb{R}}

\author{Manuele Santoprete }
\address{Department of Mathematics\\
Wilfrid Laurier University\\
75 University Avenue West\\
Waterloo, Ontario, Canada, N2L 3C5 }
\email{msantopr@wlu.ca}

\author{J\"urgen Scheurle}
\address{Zentrum Mathematik \\
TU M\"unchen\\
Boltzmannstr. 3, 85747 Garching, Germany}
\email{scheurle@ma.tum.de}

\author{Sebastian Walcher}
\address{ Lehrstuhl A f\"ur Mathematik\\
RWTH Aachen\\
52056 Aachen\\ Germany}
\email{walcher@matha.rwth-aachen.de}


\begin{document}
\title[Motion in a symmetric potential on the hyperbolic plane]{Motion in a symmetric potential \\on the hyperbolic plane}

\begin{abstract}
We study the motion of a particle in the hyperbolic plane (embedded in Minkowski space), under the action of a potential that depends only on one variable.  This problem is  the analogous to the spherical pendulum in a unidirectional force field. However, for the discussion of the hyperbolic plane one has to distinguish three inequivalent cases, depending on the direction of the force field. Symmetry reduction, with respect to groups that are not necessarily compact or even reductive, is carried out by way of Poisson varieties and Hilbert maps. For each case the dynamics is discussed, with special attention to linear potentials.
 
\end{abstract}
\maketitle

\section{Introduction}

We discuss the motion of a particle on the hyperbolic plane under the influence of a unidirectional conservative force field. Among the various realizations of the hyperbolic plane, the embedding as a hyperboloid shell in three-dimensional (Minkowski) space seems most convenient for the discussion.\\
The problem has some similarities with the spherical pendulum in a gravitational field, and our work is motivated by the detailed discussion of the spherical pendulum in Cushman and Bates \cite{Cushman}, i.e.  of the motion of a particle on a sphere in Euclidean space, under the influence of a linear potential. But there are some aspects which render the present problem more complicated. In particular, the direction of the axis for the potential is relevant, which leads to three inequivalent cases, with different types of isotropy groups, including unipotent ones. For the symmetry reduction it seems therefore more appropriate and convenient to use an alternative approach - via Poisson varieties - which works in a quite uncomplicated manner for the various cases of our ``hyperbolic pendulum". The approach, which has been used before in special cases (see \cite{WalcherNF}), is presented generally in the Appendix of the present paper. Its range of applicability covers algebraic group actions on algebraic varieties, and it is relatively uncomplicated from a technical and computational perspective. Moreover, such group actions occur frequently enough in applications to justify a general presentation of the underlying theory. The specific applications to particle motion on the hyperbolic plane will be discussed in detail.

\section{Preliminaries}
We first review some facts about the hyperbolic plane. The non-degenerate symmetric bilinear form $\langle\cdot,\cdot \rangle_L$ on  $\mathbb R^{3}$, given by
\[
\langle u, v\rangle_L =u_1v_1+u_2v_2-u_3v_3
\]
for $u=(u_1,u_2,u_3)$, and $v=(v_1,v_2,v_3)$,
is called the Lorentz (or Minkowski) pseudo inner product on $\mathbb R^{3}$. As usual we denote this pseudo inner product space as $(\mathbb R^{3},\langle\cdot,\cdot\rangle_L)$ by
$\mathbb R^{2,1}$. 
The Lorentz cross product on $\mathbb R^{2,1}$  is defined by 
\[
u\times_L v=(u_2 v_3-u_3 v_2, u_3 v_1-u_1 v_3, u_2 v_1-u_1 v_2).
\]
Moreover we define  the  Lorentz gradient for a scalar field $f$ as

\[\nabla_L f=\left(\frac{\pd f}{\pd x_1},\frac{\pd f}{\pd x_2},-\frac{\pd f}{\pd x_3}\right)^{\rm tr},\]
thus we have the identity
\begin{equation}\label{lograd}
\langle\nabla_Lf(x),\,v\rangle_L=Df(x) v
\end{equation}

 The Lorentz pseudo-inner product induces a  pseudo-Riemannian metric on $\R^{2,1}$.
Pulling back the canonical symplectic 2-form on the cotangent bundle $T^*\R^{2,1}$ (see Cushman and Bates \cite{Cushman} Appendix A, Section 2)
one obtains the symplectic form 
\[
\omega_L((u,\alpha),(v,\beta))=\langle \beta,u \rangle_L-\langle \alpha,v \rangle_L
\]
on the tangent bundle $T\R^{2,1}$, where $u,v,\alpha,\beta\in\R^{2,1}$.

From the definition of the symplectic form $\omega_L$ it follows that the  Hamiltonian vector field of a function $f$ is given by  $X_f=(\nabla_L^y f,-\nabla_L^x f)$. Hence the Poisson bracket associated with $\omega_L$ is
\beq\label{fullpoiss}
\{f,g\}(z):=\omega_L(z)(X_f(z),X_g(z))=\langle\nabla_L^x f,\nabla_L^y g\rangle_L-\langle\nabla_L^xg,\nabla_L^yf\rangle_L
\eeq
with $z=(x,y)$.

The unconstrained Hamiltonian system  $(H, T\mathbb{R}^{2,1},\omega_L)$ with 
\beq
H=\frac 1 2\langle y,y \rangle_L+U(x)
\label{Hamiltonian}
\eeq
describes a particle that moves in $\mathbb{R}^{2,1}$ under the influence of the potential  $U(x)$.
The Hamiltonian $H$ determines the Hamiltonian vector field whose integral curves satisfy
\beq
\begin{split}
 \frac{dx}{dt}&=\nabla_L^y H=\{x,H\}=y\\
\frac{dy}{dt}&=-\nabla_L^x H=\{y,H\}=-\nabla_L^x U(x).\\
\end{split}
\eeq
We now constrain the particle motion to the hyperbolic plane 
\[
\mathbb{H}^2=\{x\in \mathbb{R}^{2,1}:\langle x,x\rangle_L=-1,\mbox{ and } x_3>0\},
\]
with tangent bundle 
\[
T\mathbb{H}^2=\{(x,y)\in T\mathbb{R}^{2,1}|\langle x,x\rangle_L +1=0,~ x_3>0, \mbox{ and } \langle x, y\rangle_L=0\}.
\]
To account for the constraint and to compute the constrained vector field $H_{T\mathbb H^2}$
with respect to the symplectic form $\omega_L$,  we introduce the Dirac-Poisson bracket $\{\cdot,\cdot\}^*$
on $T\mathbb R^{2,1}$ (see Cushman and Bates \cite{Cushman}, Appendix A, Section 4).

\begin{lemma}\label{dirac}
 The Dirac-Poisson structure $\{\cdot,\cdot\}^*|_{T\mathbb H^2}$ is given by

\begin{tabular}{c|cccccc}
&$x_1$&$x_2$&$x_3$&$y_1$&$y_2$&$y_3$\\
\hline 
$x_1$&0&0&0&$1+x_1^2$&$x_1x_2$&$x_1x_3$\\ 
$x_2$& &0&0&$x_1x_2$&$1+x_2^2$&$x_2x_3$\\
$x_3$& & &0&$x_1x_3$&$x_2x_3$&$-1+x_3^2$\\ 
$y_1$& & & &0&$x_1y_2-x_2y_1$&$x_1y_3-y_1x_3$\\ 
$y_2$& & & & &0&$x_2y_3-y_2x_3$\\ 
$y_3$& & & & & &0\\ 
\end{tabular} 

with Casimirs $c_1(x,y)=\langle x, x\rangle_L+1$ and $c_2(x,y)=\langle x,y\rangle_L$.
\end{lemma}
\begin{proof}
 The phase space $T\mathbb H^2$ is given as a subset of $T\mathbb R^{2,1}$ by the constraints
\[
c_1(x,y)=\langle x, x\rangle_L+1=0 \mbox{ and } c_2(x,y)=\langle x,y\rangle_L=0.
\]
According to \cite{Cushman} the Dirac-Poisson brackets are given by the relation
\[
\{F,G\}^*=\{F,G\}+\sum_{i,j} C_{ij}\{F,c_i\}\{G,c_j\},
\]
where $C_{ij}$ are the entries of the inverse  of the matrix $\left(\{c_i,c_j\}\right)_{i,j}$. In our case 
\beq
C=\frac{1}{2\langle x,x\rangle_L}\begin{pmatrix}
0&-1\\
1&0                                 
                            \end{pmatrix}
\eeq
 The assertion now follows by straightforward computations.   
\end{proof}
Computing the equations of motion with respect to the Dirac-Poisson bracket one obtains:
\begin{lemma}
 The constrained equations on $T\mathbb H^2$ are given by
\beq
\begin{split}
 \frac{dx}{dt}&=y\\
\frac{dy}{dt}&=-\nabla_L^x U+x(\langle y,y\rangle_L-\langle x,\nabla_L^x U\rangle_L)
\end{split}
\label{motion}
\eeq
\end{lemma}
\begin{remark} The stationary points $(x^*,y^*)$ of \eqref{motion} are given by
$y^*=0$ and $\nabla_L^x U(x^*)=0$, as can be readily verified.

\end{remark}

Next we recall some facts about linear automorphisms of the hyperbolic plane. For details and proofs we refer to  \cite{Dillen} and \cite{Nomizu}. A linear map $A\in GL(\mathbb{R}^{2,1})  $ is called
orthogonal with respect to the pseudo inner product if $\langle Au,Av\rangle_L=\langle u, v\rangle_L$ for all
$u,v \in\mathbb R^{2,1}$. 
Let $O(2,1)$ denote the set of these orthogonal linear maps. For $A\in O(2,1)$ one has $(\det A)^2=1$, so $\det A=\pm 1$. One furthermore defines the special orthogonal group
$SO(2,1)=\{A\in O(2,1)|\det A=1\}$,
and finally the Lorentz group  ${\rm Lor}(2,1)$ of $\mathbb R^{2,1}$ to be the (closed) subgroup of $SO(2,1)$ preserving the hyperbolic plane $\mathbb H^2$. It can be shown that  ${\rm Lor}(2,1)$ is connected.
The action of $O(2,1)$ on $\mathbb{R}^{2,1}$ extends to a ``diagonal" action on the tangent bundle $T(\mathbb{R}^{2,1})$ via the representation
\[
A\mapsto\widetilde A:= \left(\begin{array}{cc} A&0\\
                                                                       0&A\end{array}\right).
\]
 
There are three types of 1-parameter subgroups of ${\rm Lor}(2,1)$ which leave a line pointwise fixed. We consider representatives of these types in the following.

The subgroup of {\it elliptic rotations} is given by
$$
G_e=\left\{
\begin{pmatrix}
\cos\theta & -\sin\theta & 0 \\ 
\sin\theta & \cos\theta & 0 \\ 
0 & 0 & 1 
\end{pmatrix},\quad \theta\in [0,2\pi)\right\}.
$$
These transformations leave the $x_3$-axis pointwise invariant.

The subgroup of {\it hyperbolic rotations} is given by
$$
G_h=\left\{\begin{pmatrix}
1 & 0 & 0 \\ 
0 & \cosh s & \sinh s \\ 
0 & \sinh s & \cosh s 
\end{pmatrix}\quad   -\infty<s<\infty 
  \right\}.
$$
These transformations leave the $x_1$-axis pointwise invariant.

The subgroup of {\it parabolic rotations} is given by 
$$
G_p=\left\{\begin{pmatrix}
1 & -t & t \\ 
t & 1-t^2/2 & t^2/2 \\ 
t & -t^2/2 & 1+t^2/2 
\end{pmatrix}, \quad -\infty<t<\infty
\right\}.
$$
These transformations leave the line $x_1=0$, $x_2=x_3$ pointwise invariant.

\medskip
\noindent
The {\it Principal Axis Theorem} for  ${\rm Lor}(2,1)$ states that every  transformation which leaves a line pointwise fixed  is conjugate to an element of either $G_e$, $G_h$ or $G_p$.

The following result shows that the transformations in $O(2,1)$, resp. ${\rm Lor}(2,1)$, induce Poisson maps on $T\mathbb R^{2,1}$, resp. $T\mathbb H^2$.
\begin{lemma}
{\em (a)} Consider $T\R^{2,1}$ with the symplectic form $\omega_L$, the corresponding Poisson brackets $\{,\}$, 
and the action of the orthogonal group $O(2,1)$.  For all scalar fields $f$ and $g$ and all $A\in O(2,1)$ one has
\[
\{f\circ\widetilde A,g\circ\widetilde A\}=\{f,g\}\circ\widetilde A.
\]
In particular, if $f$ and $g$ are invariant with respect to $O(2,1)$ then so is $\left\{f,g\right\}$.

\noindent{\em(b)} Furthermore, consider the constraint to the cosymplectic submanifold $T\mathbb H^2$, with the Dirac-Poisson bracket. Then for all scalar fields $f$ and $g$ and all $A\in{\rm Lor}(2,1)$ one has
\[
\{f\circ\widetilde A, g\circ\widetilde A\}^*=\{f,g\}^*\circ\widetilde A.
\]
In particular, if $f$ and $g$ are invariant with respect to ${\rm Lor}(2,1)$ then so is $\left\{f,g\right\}^*$.
\end{lemma}
\begin{proof}
(i) Let $B\in O(2,1)$. By the chain rule and \eqref{lograd} we have for every scalar field $q$ on $\mathbb R^{2,1}$ the identities
\[
\begin{array}{rcl}
\langle\left(\nabla_L(q\circ B)\right)(x),\,v\rangle_L&=& Dq(Bx)Bv\\
   &=&  \langle\left(\nabla_L(q)\right)(Bx),\,Bv\rangle_L\\
    &=&  \langle B^*\left(\nabla_L(q)\right)(Bx),\,v\rangle_L,
\end{array}
\]
with $B^*$ the adjoint of $B$.

\noindent(ii) Now let $f$ and $g$ be scalar fields on $T\mathbb R^{2,1}$. Since $A\in O(2,1)$ acts diagonally, using the result from (i) for partial derivatives yields
\[
\begin{array}{rcl}
\langle\nabla_L^x(f\circ\widetilde A),\nabla_L^y(g\circ\widetilde A)\rangle_L&=& 
\langle A^*\nabla_L^{Ax}f,A^*\nabla_L^{Ay}g\rangle_L\\
 &=& \langle\nabla_L^{Ax}f,\nabla_L^{Ay}g\rangle_L\\
 &=& \langle\nabla_L^xf,\nabla_L^yg\rangle_L\circ \widetilde A,\\
\end{array}
\]
where the invariance of the pseudo inner product has been used. Using the same line of argument once more, one sees from \eqref{fullpoiss} that the property asserted in (a) holds.

\noindent(iii) By part (a), and since the constraints $c_1$ and $c_2$ are given by invariants of ${\rm Lor}(2,1)$, the invariance property (b) follows for the Dirac-Poisson brackets.
\end{proof}

\section{Axially symmetric potentials}
In this section we discuss motion on $\mathbb H^2$ under the influence of an axisymmetric potential, including symmetry reduction and reconstruction of invariant sets from the reduced systems. By the Principal Axis Theorem for the Lorenz group it is sufficient to discuss the three subgroups $G_e,\,G_h,\,G_p$ which were introduced in the previous section. We will employ a reduction method which differs from the one presented in Cushman and Bates \cite{Cushman}, Appendix B. Our approach to reduction is based on the observation that the symmetry is given by the regular action of a linear algebraic group, rather than viewing this as a smooth Lie group action as discussed in \cite{Cushman}. Using some standard Algebraic Geometry, we obtain a reducing map and a ``reduced space" with relatively little technical expenditure. The approach, which is outlined in some detail in the Appendix, works via construction of an induced Poisson bracket on an algebraic variety, the variety being defined by the relations between the entries of a Hilbert map for the subgroup. Initially we will consider only polynomial (or rational) vector fields, and address the extension to the analytic or smooth case later. According to Proposition \ref{poiredprop} the reduction starts from an algebraic (Poisson) variety, hence we will consider vector fields on the variety
\[
V:=\left\{ (x,y)\in T\mathbb{R}^{2,1}|\langle x,x\rangle_L +1= 0 \mbox{ and } \langle x, y\rangle_L=0\right\}\supseteq \mathbb H^2,
\]
with no loss of generality.

The invariant algebra $\mathbb R[V]^G$ is finitely generated when the group $G$ is compact or, more generally, reductive (i.e.,  every finite dimensional linear representation is completely reducible), according to a famous result from invariant theory (see e.g. Springer \cite{Springer} for affine $n$-space, and Haboush \cite{Haboush} for general affine varieties).
In the elliptic case we have a compact symmetry group, and our reduction will turn out to be essentially the singular reduction described in \cite{Cushman}. In the hyperbolic case the symmetry group is no longer compact but still reductive. In the parabolic case one has to deal with a unipotent symmetry group. 

The computation of invariants in the first two cases is facilitated by some (well-known) auxiliary results. 
\begin{lemma}\label{torus}
Let $\alpha\in\mathbb C^*$ and consider the one-parameter group $G$ with infinitesimal generator 
\[
{\rm diag}(\alpha, -\alpha,0,\alpha,-\alpha,0)
\]
acting on $\mathbb C^6$ with coordinates $z_1,\ldots,z_6$.
The invariant algebra of $G$ is generated by 
\[
\psi_1=z_1z_2,\,\psi_2=z_3,\,\psi_3=z_4z_5,\,\psi_4=z_6,\,\psi_5=z_1z_5,\,\psi_6=z_2z_4.
\]
One has the relation $\psi_1\psi_3-\psi_5\psi_6=0$ between the generators; all other relations are a consequence of this one.
\end{lemma}
\begin{proof}[Sketch of proof] By e.g. \cite{WalcherNF}, Lemma 1.2 ff. every invariant is a linear combination of monomials $\prod \,z_i^{m_i}$ such that 
\[
\alpha\cdot (m_1-m_2)+0\cdot m_3+\alpha\cdot(m_4-m_5)+0\cdot m_6=0.
\]
The assertion follows by determining the nonnegative integer solutions of this equation.
\end{proof}

\begin{lemma}\label{panlem}  Let $G$ be a reductive group which acts on $\mathbb R^n$ or $\mathbb C^n$ and stabilizes the variety $V$. Then all elements of $\mathbb R[V]^G$ are restrictions of elements of $\mathbb R[x_1,\ldots,x_n]^G$ to $V$.
\end{lemma}
\begin{proof} See e.g. Panyushev \cite{Panyushev}, Lemma.
\end{proof}

\subsection{Elliptic rotations}
In this subsection we discuss potentials that depend only on $x_3$, i.e. $U(x)=U(x_3)$. Such potentials are invariant with respect to the compact group $G=G_e$ of elliptic rotations
$$R_e(\theta)=\begin{pmatrix}
\cos\theta & -\sin\theta & 0 \\ 
\sin\theta & \cos\theta & 0 \\ 
0 & 0 & 1 
\end{pmatrix} 
$$  
about the $x_3$ axis.  For this action the singular reduction theory developed in Cushman and Bates \cite{Cushman}
is applicable, but we will take the alternative approach outlined in the Appendix.

The tangent lift of $R_e(\theta)$ gives an action of $G\cong S^1$ on $T\mathbb H^2$
\[
\Phi_e: S^1\times T\mathbb H^2\to T\mathbb H^2:(\theta, x,y)\to (R_e(\theta)x, R_e(\theta)y),
\]
with infinitesimal generator $Y_e$ whose integral curves satisfy

\beq
\begin{split}
 \frac{dx}{dt}&=-x\times_L e_3\\
 \frac{dy}{dt}&=-y\times_L e_3\\
\end{split}
\eeq
Obviously $Y_e$ is the Hamiltonian vector field corresponding to the Hamiltonian
\beq
J_e: T\mathbb{R}^{2,1}\to \mathbb{R}:(x,y)\to \langle x\times_L y, e_3\rangle_L=x_1y_2-x_2y_1,
\eeq
and $J_e$ is the momentum mapping of the action $\Phi_e$. One can check that
\[
\{J_e,H\}^*|_{T\mathbb H^2}=0.
\]
that is, $J_e|{T\mathbb H^2}$ is an integral of motion. In particular
 the Hamiltonian system that describes the motion of a particle on the hyperbolic plane under the influence of an axisymmetric potential $U=U(x_3)$ is Liouville integrable.

Now we turn to symmetry reduction. 
\begin{lemma}\label{ellipticgen}
{\em(a)} The polynomial ring  $\mathbb{R}[x,y]^{G}$ of  polynomials in $(x,y)$  invariant under the group action is generated by 
\begin{align*}
\sigma_1 &= x_3 &\sigma_2 &= y_3 & \sigma_3 &=y_1^2+y_2^2-y_3^2 \\
\sigma_4 &=x_1y_1+x_2y_2    &\sigma_5 &=x_1^2+x_2^2     & \sigma_6 &=x_1y_2-x_2y_1
\end{align*}
subject to the relation
\beq
\sigma_4^2+\sigma_6^2=\sigma_5(\sigma_3+\sigma_2^2).
\label{relations-e}
\eeq

\noindent{\em(b)} The ring $\mathbb R[V]^{G}$ of regular functions invariant under the group action is generated by the 
restrictions of the $\sigma_j$ to $V$, with additional relations
\beq
\sigma_5-\sigma_1^2=-1, \quad \mbox{ and }\sigma_4-\sigma_1\sigma_2=0.
\label{restriction-e}
\eeq
Thus the restrictions of $\sigma_1$, $\sigma_2$, $\sigma_3$ and $\sigma_6$ to $V$ generate $\mathbb{R}[x,y]^{G}$.
\label{HilbertBasis(elliptic)}
\end{lemma}
\begin{proof}[Sketch of proof]
To prove (a), diagonalize the infinitesimal generator of the $R_e(\theta)$ over $\mathbb C$, and pass to the lifted action to obtain the matrix from Lemma \ref{torus}, with $\alpha =i$. Going back to real coordinates, the assertion follows. (The generator $\sigma_3$ was chosen, instead of $\hat\sigma_3=y_1^2+y_2^2=\sigma_3+\sigma_2^2$, for reasons of convenience.) Part (b) is a direct consequence of Lemma \ref{panlem}.
\end{proof}
\begin{proposition}\label{ellipticredprop}
{\em(a)} The orbit of $(0,0,1,0,0,0)^{\rm tr}$ under the group action consists of this point only. The orbit of every other point on $T\mathbb H^2$ is an ellipse which lies in a two-dimensional affine subspace of $\mathbb R^6$.\\
{\em(b)} The map 
\[
\Gamma:\,V\to \mathbb R^4,\quad z\mapsto\left(\begin{array}{c} \sigma_1(z)\\
                                                                                                    \sigma_2(z)\\
                                                                                                    \sigma_3(z)\\
                                                                                                    \sigma_6(z)\end{array}\right),
                                                                                      \]
where we denote with $ w = ( w _1, w _2 , w _3, w _4)^{tr}$ the image of $z$ under $ \Gamma $,   
sends $V$ to the variety $Y\subseteq \mathbb R^4$ which is defined by 
\begin{equation}\label{ellipvariety}
w_4^2=(w_1^2-1)w_3-w_2^2.
\end{equation}
{\em (c)} The image of $T\mathbb H^2$ under the map $ \Gamma $ is the union of
\[
\left\{w\in Y;\,w_1>1\mbox{ and }w_2^2+w_3>0\right\}
\]
and
\[
\{\left(\begin{array}{c}1\\
                                                                                                                                 0\\
                                                                                                                                 w_3\\
                                                                                                                                 0\end{array}\right); w_3\geq 0\}
\cup\{\left(\begin{array}{c}w_1\\
                                                  0\\
                                                    0\\
                                                 0\end{array}\right); w_1\geq 1\}.
\]
In particular $\Gamma(V)$ is Zariski-dense in $Y$.\\
{\em(d)} The Poisson-Dirac bracket on $V$ induces a Poisson bracket $\left\{\cdot,\cdot\right\}^\prime$ on $Y$, which is determined by 
\[
\{w_1,w_2\}^\prime=-1+w_1^2, \quad \{w_2,w_3\}^\prime=2w_1w_3, \quad \{w_3,w_1\}^\prime=-2w_2,
\]
and
\[
\quad \{w_i,w_4\}^\prime=0,\,1\leq i\leq 3.
\]
{\em (e)} The equations of motion \eqref{motion} on $\mathbb H^2$ for the Hamiltonian 
\[
H=\frac12 \sigma_3 + U(\sigma_1)
\]
with axisymmetric potential are mapped by $\Gamma$ to the reduced Hamiltonian system
\[
\begin{array}{rcccl}
\dot w_1&=&\left\{w_1,h\right\}^\prime&=& w_2\\
\dot w_2&=&\left\{w_2,h\right\}^\prime&=& w_1w_3+(1-w_1^2)U^\prime(w_1)\\
\dot w_3&=&\left\{w_3,h\right\}^\prime&=& -2w_2U^\prime(w_1)\\
\dot w_4&=&\left\{w_4,h\right\}^\prime&=& 0
\end{array}
\]
on $Y$, with $h(w):= \frac12 w_3+U(w_1)$. This system admits the first integrals $h$ and $w_4$. 
\end{proposition}
\begin{proof} We verify part (a) by elementary means. Thus let $(u,v)^{\rm tr}\in \mathbb R^6$. The points $(x,y)^{\rm tr}$ in its orbit satisfy $x_3=u_3$, $y_3=v_3$ and
\[
\left(\begin{array}{c}x_1\\
                               x_2\\
                               y_1\\
                                y_2\end{array}\right)=\cos\theta \left(\begin{array}{c}u_1\\
                               u_2\\
                               v_1\\
                                v_2\end{array}\right) +\sin\theta\left(\begin{array}{c}-u_2\\
                               u_1\\
                               -v_2\\
                                v_1\end{array}\right).
\]
This is obviously an ellipse in a two-dimensional affine subspace (generated by $ (u _1, u _2 , v _1 ,v _2 )^{tr}$ and $ ( -u _2 , u _1 , -v _2 , v _1 )^{tr} $ )  unless $u_1=u_2=v_1=v_2=0$. In the latter case we find $x_3=1$ and $y_3=0$ from the conditions defining $T\mathbb H^2$.\\
Combining \eqref{relations-e} and \eqref{restriction-e} yields
\[
\sigma_6^2=(\sigma_1^2-1)\sigma_3-\sigma_2^2,
\]
which implies $\Gamma(V)\subseteq Y$, and (b) follows.\\
To prove part (c), let $w\in\Gamma(T\mathbb H^2)$, and $(x,y)^{\rm tr}$ in its inverse image. Then
\[
\begin{array}{rcl}
x_3&=& w_1\\
y_3&=& w_2\\
y_1^2+y_2^2-y_3^2&=& w_3\\
x_1y_2-x_2y_1&=& w_4
\end{array}
\]
and the conditions defining $T\mathbb H^2$ immediately show that $w_1\geq 1$ and $w_2^2+w_3\geq 0$.\\
If $w_1>1$ and $w_2^2+w_3>0$ then one obtains an inverse image as follows: Choose $x_3=w_1$, $y_3=w_2$, $y_1=0$, $y_2=\sqrt{w_2^2+w_3}$, $x_1=w_4/\sqrt{w_2^2+w_3}$. From equation \eqref{ellipvariety} one sees that
\[
x_1^2-x_3^2+1=-w_1^2w_2^2\leq 0,
\]
and therefore one can find $x_2$ such that $x_1^2+x_2^2-x_3^2=-1$, whence $(x,y)^{\rm tr}\in T\mathbb H^2$.\\
If $w_1=1$ then equation \eqref{ellipvariety} shows $x_1=x_2=0$ and $y_3=0$. Thus only the points 
$(1,0,w_3,0)^{\rm tr}$ with $w_3\geq 0$ are in $\Gamma(T\mathbb H^2)$.\\
If $w_2^2+w_3=0$ then $y_1=y_2=0$, hence $w_4=0$, and $y_3=0$ from equation \eqref{ellipvariety}; thus $w_2=w_3=0$. Therefore only the points $(w_1,0,0,0)^{\rm tr}$ with $w_1\geq 1$ lie in the image of $T\mathbb H^2$.\\
Part (d) is proven by the following computations 
 (using Lemma \ref{dirac}), and invoking Proposition \ref{poiredprop}.
\[
\begin{array}{rclcl}
\{\sigma_1, \sigma_2\}^*&=&\{x_3,y_3\}^*&=&x_3^2-1=\sigma_1^2-1\\
\{\sigma_3, \sigma_1\}^*&=&\{y_1^2+y_2^2-y_3^2,x_3\}^*&=&2y_1(-x_1x_3)+2y_2(-x_2x_3)-2y_3(1-x_3^2)\\
                              & & &=& -2x_3(x_1y_1+x_2y_2-x_3y_3) -2y_3\\
                              & & &=& -2\sigma_2
\end{array}
\]
In the same manner one shows $\{\sigma_2,\sigma_3\}^*=2\sigma_1\sigma_3$ and all $\{\sigma_6,\sigma_i\}^*=0$. Part (e) now follows from straighforward computations, and again by invoking Proposition \ref{poiredprop}.
\end{proof}
Note that the theorem of Procesi and Schwarz \cite{ProcesiSchwarz} is not directly applicable to $\Gamma$; thus the inequalities determining the image need to be obtained directly.\\

\medskip
Rather than discussing the reduced system on $Y$, one may consider the map
\[
\widetilde \Gamma:\,V\to \mathbb R^3,\quad z\mapsto\left(\begin{array}{c} \sigma_1(z)\\
                                                                                                    \sigma_2(z)\\
                                                                                                    \sigma_3(z)
                                                                                                    \end{array}\right),
                                                                                      \]
thus compose $\Gamma$ and the projection $\mathbb R^4\to \mathbb R^3$ onto the first three entries. This means restricting the system in part (e) of Proposition \ref{ellipticredprop} to the
system for $w_1$, $w_2$ and $w_3$ on $\mathbb R^3$, with first integrals $j^2=(w_1^2-1)w_3-w_2^2$ and $h$. 

\begin{corollary}\label{ellipticredmap}
{\em (a)} A point $w=(w_1,\,w_2,\,w_3)^{\rm tr}\in\mathbb R^3$ lies in $\widetilde\Gamma(T\mathbb H^2)$ if and only if the following inequalities are satisfied:
\[
w_1\geq 1,\quad w_2^2+w_3\geq 0,\quad (w_1^2-1)w_3-w_2^2\geq 0.
\]
{\em(b)} The inverse image of $w$ is a union of two ellipses whenever 
$(w_1^2-1)w_3-w_2^2>0$, an ellipse whenever $(w_1^2-1)w_3-w_2^2=0$ and $(w_1,w_2^2+w_3)\not=(1,0)$, and equal to the point $(0,0,1,0,0,0)$ if  $(w_1,w_2^2+w_3)=(1,0)$, i.e., $w=(1,0,0)$.

\end{corollary}
\begin{proof} Since $G$ is compact, the inverse image of every point in $\Gamma(T\mathbb H^2)$ is a group orbit (see \cite{Cushman}), and these were determined in Proposition \ref{ellipticredprop}. Moreover, by \eqref{ellipvariety}
a point of $\mathbb R^3$ has exactly two inverse images in $Y$ under the projection
\[
\left(\begin{array}{c}w_1\\
                               w_2\\
                               w_3\\
                                w_4\end{array}\right)\mapsto \left(\begin{array}{c}w_1\\
                               w_2\\
                               w_3\end{array}\right)
\]
if and only if $(w_1^2-1)w_3-w_2^2>0$, and exactly one inverse image if and only if $(w_1^2-1)w_3-w_2^2=0$.  With Proposition \ref{ellipticredprop} all assertions follow by routine verification.
\end{proof}
The discussion of the reduced system
\begin{equation}\label{ellipticredsys}
\begin{array}{rcl}
\dot w_1&=& w_2\\
\dot w_2&=& w_1w_3+(1-w_1^2)U^\prime(w_1)\\
\dot w_3&=& -2w_2U^\prime(w_1)
\end{array}
\end{equation}
on $\widetilde\Gamma(T\mathbb H^2)$ is now relatively straightforward, invoking some familiar facts about one-dimensional manifolds and differential equations. We will
discuss the system for rational potential $U$.
\begin{proposition}\label{ellipticredana}
{\em (a)} Relative equilibria: The stationary points of \eqref{ellipticredsys} on $\widetilde\Gamma(T\mathbb H^2)$ are 
\[
z_1:=\left(\begin{array}{c}1\\
                               0\\
                                0\end{array}\right)\quad{\rm and}\quad z_\rho:=\left(\begin{array}{c}\rho\\
                               0\\
                                \frac{\rho^2-1}{\rho}U^\prime(\rho)\end{array}\right), \quad{\rm with}\,\, \rho>1,\,\,U^\prime(\rho)\geq 0.
\]
The inverse image of $z_1$ under $\widetilde\Gamma$ consists of the stationary point $(0,0,1,0,0,0)$. If $\rho>1$ and $U^\prime(\rho)>0$ then the inverse image of $z_\rho$ under $\widetilde\Gamma$ is a union of two ellipses, each of which is a closed trajectory of \eqref{motion}. If $\rho>1$ and $U^\prime(\rho)=0$ then the inverse image of $z_\rho$ is an ellipse consisting of stationary points of \eqref{motion}.\\
{\em(b)} Every nonstationary solution of \eqref{ellipticredsys} is contained in the subvariety of $\mathbb R^3$ defined by the equations
\[
\begin{array}{rclcl}
j^2&:=&(w_1^2-1)w_3-w_2^2&=& c_1\\
h&:=&\frac12 w_3+U(w_1)&=& c_2
\end{array}
\]
with constants $c_1\geq 0$ and $c_2$. This variety is smooth, hence a one-dimensional submanifold, whenever it contains no stationary point. In this case, each of its connected components is diffeomorphic to either a circle or a line, and is a trajectory of \eqref{ellipticredsys}.\\
{\em (c)} For $c_1=0$ only the stationary point $(1,0,0)$ and the half-line defined by $w_1-1=w_2=0$ and $w_3>0$ are in $\widetilde\Gamma(T\mathbb H^2)$. The stationary points on this half-line are precisely the $z_\rho$, $\rho>1$ with $U^\prime(\rho)=0$, and the half-line is a union of stationary points and nonstationary solutions whose limit sets are contained in the set of stationary points.\\
{\em (d)} The inverse image of a nonstationary solution of \eqref{ellipticredsys} on the level set $j=c_1>0$, $h=c_2$ is diffeomorphic to a torus or a cylinder. Moreover it is an invariant set of \eqref{motion} which contains no stationary points.\\
The inverse image of a level set $j=0$, $h=c_2$ is a half-cone (``pinched cylinder") which contains the stationary point $(0,0,1,0,0,0)$ and all circles defined by $x_3=\rho>1$, $x_1^2+x_2^2=\rho^2-1$ and $y=0$, which are made up from stationary points.
\end{proposition}
\begin{proof} The proof of part (a) is straightforward with Corollary \ref{ellipticredmap}. The Jacobian of $(j^2,h)$ is equal to
\[
\left(\begin{array}{ccc} 2w_1w_3& -2 w_2& w_1^2-1\\
                                       U^\prime(w_1) & 0 & 1/2\end{array}\right),
\]
which has rank $<2$ only if $w_2=0$ (consider the second and third columns), and $w$ is stationary (consider the first and third columns). Conversely, the Jacobian has rank $<2$ at every stationary point. Thus the variety is smooth whenever 
it contains no stationary point.
From these observations and routine arguments the remaining assertions follow.

\end{proof}
\begin{remark}(a) One can determine the level sets of $(j^2,h)$ in a direct manner. Indeed, for $w_1>1$ one finds
\[
2\left(c_2-U(w_1)\right)=w_3=\frac{w_2^2+c_1}{w_1^2-1}
\]
and thus
\[
w_2^2=2\left(c_2-U(w_1)\right)\left(w_1^2-1\right)-c_1.
\]
(b) While the reduction procedure is initially restricted to rational potentials $U$, a well-known theorem by Schwarz \cite{Schwarz} allows extension to smooth potentials; see also the Remark at the end of the Appendix.

\end{remark}
\begin{example} For linear potential, thus $U(x_3)= c\cdot x_3$ with $c\not=0$, system  \eqref{ellipticredsys} admits equilibria
\[
z_\rho=\left(\begin{array}{c} \rho\\
                                              0\\
                                          c(\rho^2-1)/\rho\end{array}\right),\quad \rho\geq 1
\]
if $c>0$. The corresponding solutions of \eqref{motion} can be interpreted as uniform motions on a circle $x_3=\rho$ whenever $\rho>1$.  For $c>0$ all trajectories of the reduced system are bounded (see Figure \ref{fig:elliptic}) , as can be seen from the expression in the Remark above. The inverse images of these trajectories are tori.
\begin{figure}
        \centering
        \begin{subfigure}[b]{0.5\textwidth}
                \centering
                \includegraphics[width=\textwidth]{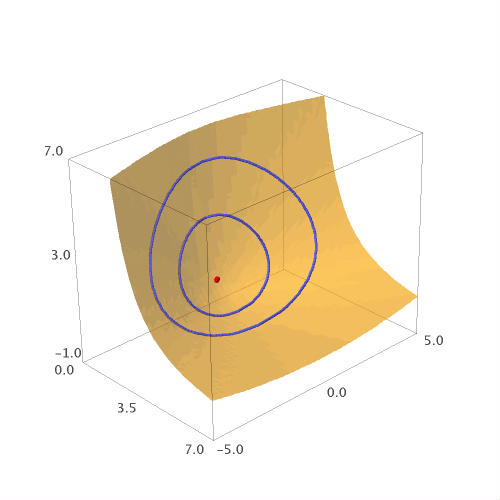}
                \caption{ $ c = 1 $ and $ c _1 = 4 $ }
        \end{subfigure}%
        ~ 
        \begin{subfigure}[b]{0.5\textwidth}
                \centering
                \includegraphics[width=\textwidth]{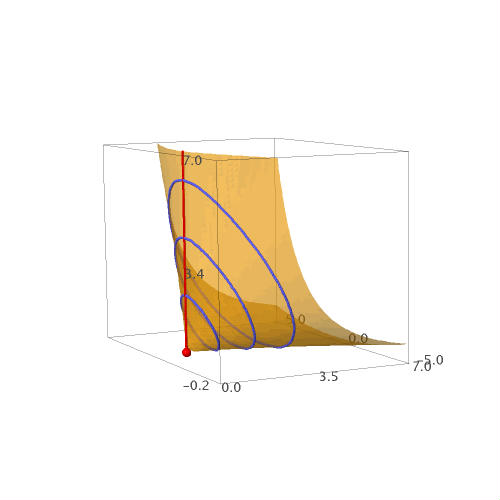}
               \caption{ $ c = 1 $ and $ c _1 = 0 $ }
        \end{subfigure} 
        \caption{Trajectories (in blue) and stationary points (in red) of the reduced system in the elliptic case  with a linear potential.}\label{fig:elliptic}
\end{figure}
In case $c<0$ the only equilibrium of \eqref{ellipticredsys} is $z_1$. All other trajectories of the system are unbounded. Their inverse images are invariant cylinders.

\end{example}
\subsection{Hyperbolic Rotations}
We now turn to the class of potentials that depend only on $x_1$, i.e. $U(x)=U(x_1)$. Such potentials are invariant with respect to the group $G=G_h\cong (\mathbb R_+^*,\,\cdot)$ of hyperbolic rotations
$$R_h(s)=\begin{pmatrix}
1 & 0 & 0 \\ 
0 & \cosh s & \sinh s \\ 
0 & \sinh s & \cosh s 
\end{pmatrix},\quad s\in\mathbb R
$$  
about the $x_1$ axis.  The tangent lift gives an action on $T\mathbb H^2$ via
\[
\Phi_h: H^1\times T\mathbb H^2\to T\mathbb H^2:(x,y)\to (R_h(t)x, R_h(t)y).
\]
This group is not compact, and its action on the six-dimensional ambient space is not proper. But $G_h$ is reductive, hence the reduction via invariants works just as well in this setting. (It turns out that the restriction of the action to $T\mathbb H^2$ is proper, but we will not use this fact, since reduction by the Poisson variety approach works conveniently.)
The infinitesimal generator of $\Phi_h$ is the vector field $Y_h$ whose integral curves satisfy

\beq
\begin{split}
 \frac{dx}{dt}&=x\times_L e_1\\
 \frac{dy}{dt}&=y\times_L e_1\\
\end{split}
\eeq
Note that $Y_h$ is the Hamiltonian vector field corresponding to the Hamiltonian
\beq
J_h: T\mathbb{R}^{2,1}\to \mathbb{R}:(x,y)\to -\langle x\times_L y, e_1\rangle_L=x_3y_2-x_2y_3,
\eeq
and $J_h$ is the momentum mapping of the action $\Phi_h$. One can check that
\[
\{J_h,H\}^*|_{T\mathbb H^2}=0.
\]
that is, $J_h|{T\mathbb H^2}$ is an integral of motion. Therefore the Hamiltonian system that describes the motion of a particle on the hyperbolic plane under the influence of an axisymmetric potential $U(x_1)$ is Liouville integrable.

We turn to symmetry reduction for this case. The proof of the following is analogous to Lemma \ref{ellipticgen}.
\begin{lemma}\label{hyperbolicgen}
 {\em (a)}The algebra $\mathbb{R}[x,y]^{G}$ of $G$-invariant polynomials is generated by 
\begin{align*}
\tau_1 &= x_1 &\tau_2 &= y_1 & \tau_3 &=y_1^2+y_2^2-y_3^2 \\
\tau_4 &=x_2y_2-x_3y_3    &\tau_5 &=x_2^2-x_3^2      & \tau_6 &=x_3y_2-x_2y_3
\end{align*}
subject to the relation
\beq
\tau_4^2-\tau_6^2=\tau_5(\tau_3-\tau_2^2).
\label{relations-h}
\eeq
{\em(b)} The ring $\mathbb R[V]^G$ of regular functions invariant under the group action is generated by the restrictions of the $\tau_j$ to $V$, with additional relations
\beq
\tau_5+\tau_1^2=-1 \mbox{ and }\tau_4+\tau_1\tau_2=0.
\label{restriction-h}
\eeq
Thus the restrictions of $\tau_1$, $\tau_2$, $\tau_3$ and $\tau_6$ to $V$ generate $\mathbb R[V]^G$.
\end{lemma}
\begin{proposition}\label{hyperbolicredprop} 
{\em(a)} The orbit of every point on $T\mathbb H^2$ is a branch of a hyperbola which lies in a two-dimensional affine subspace of $\mathbb R^6$.\\
{\em(b)} The map 
\[
\Gamma:\,V\to \mathbb R^4,\quad z\mapsto\left(\begin{array}{c} \tau_1(z)\\
                                                                                                    \tau_2(z)\\
                                                                                                    \tau_3(z)\\
                                                                                                    \tau_6(z)\end{array}\right),
                                                                                      \]
where we denote with $ w = ( w _1, w _2 , w _3, w _4)^{tr}$ the image of $z$ under $ \Gamma $, sends $V$ to the variety $Y\subseteq \mathbb R^4$ which is defined by 
\begin{equation}\label{hypvariety}
w_4^2=(w_1^2+1)w_3-w_2^2.
\end{equation}
{\em (c)} The image of $T\mathbb H^2$ under the map  $ \Gamma $ is equal to $Y$, hence
 $\Gamma(V)=Y$.\\
{\em(d)} The Poisson-Dirac bracket on $V$ induces a Poisson bracket $\left\{\cdot,\cdot\right\}^\prime$ on $Y$, which is determined by 
\[
\{w_1,w_2\}^\prime=1+w_1^2, \quad \{w_2,w_3\}^\prime=2w_1w_3, \quad \{w_3,w_1\}^\prime=-2w_2,
\]
and
\[
\quad \{w_i,w_4\}^\prime=0,\,1\leq i\leq 3.
\]
{\em (e)} The equations of motion \eqref{motion} on $\mathbb H^2$ for the Hamiltonian 
\[
H=\frac12 \tau_3 + U(\tau_1)
\]
with axisymmetric potential are mapped by $\Gamma$ to the reduced Hamiltonian system
\[
\begin{array}{rcccl}
\dot w_1&=&\left\{w_1,h\right\}^\prime&=& w_2\\
\dot w_2&=&\left\{w_2,h\right\}^\prime&=& w_1w_3-(1+w_1^2)U^\prime(w_1)\\
\dot w_3&=&\left\{w_3,h\right\}^\prime&=& -2w_2U^\prime(w_1)\\
\dot w_4&=&\left\{w_4,h\right\}^\prime&=& 0
\end{array}
\]
on $Y$, with $h(w):= \frac12 w_3+U(w_1)$. This system admits the first integrals  $h$ and $w_4$. 
\end{proposition}
\begin{proof} For part (a), let $(u,v)^{\rm tr}\in \mathbb R^6$. The points $(x,y)^{\rm tr}$ in its orbit satisfy $x_1=u_1$, $y_1=v_1$ and
\[
\left(\begin{array}{c}x_2\\
                               x_2\\
                               y_2\\
                                y_3\end{array}\right)=\cosh s\left(\begin{array}{c}u_2\\
                               u_3\\
                               v_2\\
                                v_3\end{array}\right) +\sinh s\left(\begin{array}{c} u_3\\
                               u_2\\
                               v_3\\
                                v_2\end{array}\right).
\]
This is obviously a branch of a hyperbola in a two-dimensional affine subspace whenever the two vectors on the right hand side are linearly independent. But linear independence follows from the relation $u_1^2+u_2^2-u_3^2=-1$ on $T\mathbb H^2$.\\
Part (b) follows from Lemma \ref{hyperbolicgen}.\\
To prove (c), let $w\in Y$. For an inverse image one has the conditions
\[
\begin{array}{rcl}
 x_1&=& w_1\\
y_1&=& w_2\\
y_2^2-y_3^2&=& w_3-w_2^2\\
x_3y_2-x_2y_3&=&w_4
\end{array}
\]
and moreover
\[
\begin{array}{rcl}
 x_2^2-x_3^2&=& -(1+w_1^2)\\
x_2y_2-x_3y_3&=&-w_1w_2
\end{array}
\]
Now choose $x_1=w_1$, $y_1=w_2$, $x_2=0$, $x_3=\sqrt{1+w_1^2}$, $y_3=w_1w_2/\sqrt{1+w_1^2}$. To find an inverse image one needs $y_2$ such that
\[
\begin{array}{rcl}
 y_2^2&=&w_3-w_2^2-w_1^2w_2^2/(1+w_1^2)\\
&=& (1+w_1^2)^{-1}\left((1+w_1^2)(w_3-w_2^2)-w_1^2w_2^2\right)\\
 &=& (1+w_1^2)^{-1}w_4^2.
\end{array}
\]
But the last expression is $\geq 0$ in view of \eqref{hypvariety}, and from this one obtains an inverse image.\\
The remainder of the proof is analogous to Proposition  \ref{ellipticredprop}.
\end{proof}
Similar to the elliptic case, we further consider the map
\[
\widetilde\Gamma:\, V\to \mathbb R^3,\quad z\mapsto\left(\begin{array}{c} \tau_1(z)\\
                                                                                                    \tau_2(z)\\
                                                                                                    \tau_3(z)
                                                                                                    \end{array}\right),
\]
composing $\Gamma$ and the projection from $\mathbb R^4$ to $\mathbb R^3$ onto the first three entries. This means restricting the system in part (e) of Proposition \ref{hyperbolicredprop} to the
system on $\mathbb R^3$ for $w_1$, $w_2$ and $w_3$, with first integrals $j^2=(w_1^2+1)w_3-w_2^2$ and $h$. 
\begin{corollary}\label{hyperbolicredmap}
{\em (a)} A point $w=(w_1,\,w_2,\,w_3)^{\rm tr}\in\mathbb R^3$ lies in $\widetilde\Gamma(T\mathbb H^2)$ if and only if 
\[
 (w_1^2+1)w_3-w_2^2\geq 0.
\]
{\em(b)} The inverse image of $w$ is the union of two hyperbola branches whenever $(w_1^2+1)w_3-w_2^2>0$, and a branch of a hyperbola whenever $(w_1^2+1)w_3-w_2^2=0$ .
\end{corollary}
\begin{proof} Part (a) is a direct consequence of Proposition \ref{hyperbolicredprop}. To prove part (b), one can invoke some theory to avoid computations. As shown in Proposition \ref{hyperbolicredprop}, the group orbits on $T\mathbb H^2$ are closed in the norm topology. This implies, by Birkes \cite{Birkes}, that the corresponding orbits in the complexification of $V$ are closed in the Zariski topology. Due to a result of Luna \cite{Luna2} the closed orbits on the complexification of $V$ stand in 1-1 correspondence with the points of the complexification of $Y$, via $\Gamma$. Going back to the real setting, one has a 1-1 correspondence between the orbits on $T\mathbb H^2$ and points on $Y$.
\end{proof}
The reduced equations of motions on $Y=\widetilde\Gamma(T\mathbb H^2)$ are 
\begin{equation}\label{hyperbolicredsys}
\begin{array}{rcl}
\dot w_1&=& w_2\\
\dot w_2&=& w_1w_3-(1+w_1^2)U^\prime(w_1)\\
\dot w_3&=& -2w_2U^\prime(w_1)
\end{array}
\end{equation}
with inequality $(w_1+1)^2w_3-w_2^2\geq 0$. We will again discuss them for rational potential $U$. The proof is a straightforward variant of Proposition \ref{ellipticredana}.

\begin{proposition}\label{hyperbolicredana}
{\em (a)} Relative equilibria: The stationary points of \eqref{hyperbolicredsys} are given by
\[
 z_\rho:=\left(\begin{array}{c}\rho\\
                               0\\
                                \frac{\rho^2+1}{\rho}U^\prime(\rho)\end{array}\right)
\]
with either $\rho>0$ and $U^\prime(\rho)\geq 0$, or  $\rho<0$ and $U^\prime(\rho)\leq 0$. In the special case $U^\prime(0)=0$ there are additional stationary points
\[
z^*_\sigma= \left(\begin{array}{c}0\\
                               0\\
                                \sigma\end{array}\right), \quad \sigma\geq 0.
\]
The inverse image of $z_\rho$ under $\widetilde\Gamma$ is a branch of a hyperbola, which is a trajectory of \eqref{motion} in case $U^\prime(\rho)\not=0$, and consists of stationary points only in case $U^\prime(\rho)=0$. The inverse image of $z^*_\sigma$ is a branch of a hyperbola, which consists of stationary points only.\\
{\em(b)} Every nonstationary solution of \eqref{hyperbolicredsys} is contained in a level set
\[
\begin{array}{rclcl}
j^2&:=&(w_1^2+1)w_3-w_2^2&=& c_1\\
h&:=&\frac12 w_3+U(w_1)&=& c_2
\end{array}
\]
with real constants $c_1\geq 0$ and $c_2$, which is smooth at every nonstationary point of the reduced system. If this level set contains no stationary point then it is a one-dimensional submanifold and every connected component is a trajectory of \eqref{hyperbolicredsys}. If the level set contains a stationary point then the complement of the stationary point set is a union of one-dimensional submanifolds which are orbits, with limit sets contained in the set of stationary points.\\
{\em (c)} The inverse image of a nonstationary solution of \eqref{hyperbolicredsys} on the level set $j^2=c_1$, $h=c_2$ is an invariant set of \eqref{motion} which is diffeomorphic to a cylinder if the solution of the reduced system is bounded, and diffeomorphic to a plane otherwise.
\end{proposition}
\begin{remark}(a) One can determine the level sets of $(j^2,h)$ directly from
\[
2\left(c_2-U(w_1)\right)=w_3=\frac{w_2^2+c_1}{w_1^2+1}
\]
and thus
\[
w_2^2=2\left(c_2-U(w_1)\right)\left(w_1^2+1\right)-c_1.
\]
(b) While the reduction procedure is initially restricted to rational potentials $U$, a theorem by Luna \cite{Luna} allows extension to analytic potentials; see also the Remark at the end of the Appendix.
\end{remark}
\begin{example} For linear potential, thus $U(x_1)= c\cdot x_1$ with $c\not=0$, system  \eqref{hyperbolicredsys} admits equilibria
\[
z_\rho=\left(\begin{array}{c} \rho\\
                                              0\\
                                          c(\rho^2+1)/\rho\end{array}\right), \quad\rho\not=0,\,c\rho>0.
\]
\begin{figure}[h]
   \centering
    \includegraphics[width=0.7\textwidth]{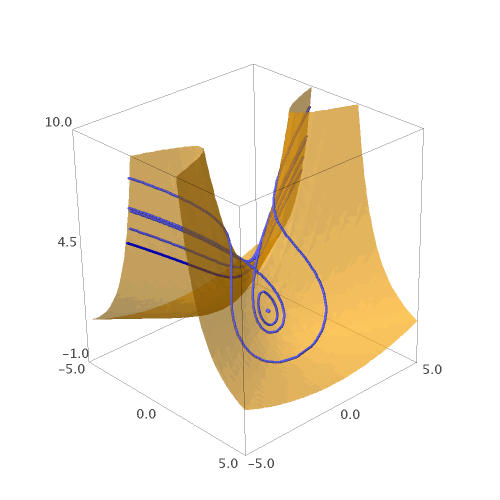}
 \caption{Trajectories of the reduced system in the hyperbolic case  with a linear potential with  $ c = 1 $ and $ c _1 = 4 $ \label{fig:hyperbolic_case}}
\end{figure}
The corresponding solutions of \eqref{motion} can be interpreted as motion on a hyperbola branch. 
We will discuss the trajectories of  \eqref{hyperbolicredsys} in some detail, starting from the conditions
\[
\begin{array}{c}
w_2^2=2(w_1^2+1)\left(c_2-cw_1\right)-c_1,\\
     w_3=2\left(c_2-cw_1\right).
\end{array}
\]
\begin{figure}[h!]
   \centering
    \includegraphics[width=0.7\textwidth]{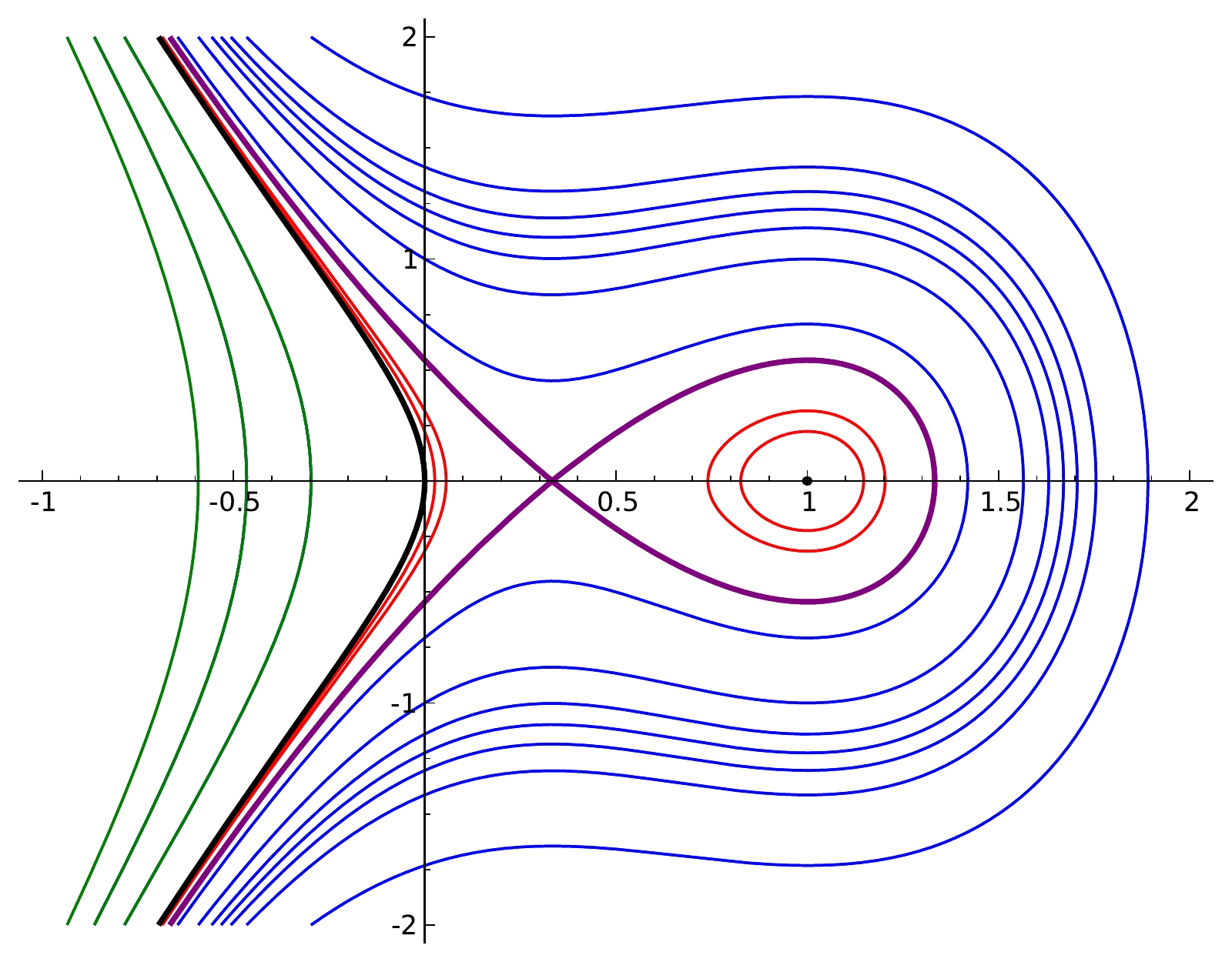}
 \caption{Example of the sets of points satisfying the equation  $w_2^2=2(w_1^2+1)\left(c_2-cw_1\right)-c_1$ in Case 3.\label{fig:case3}}
\end{figure}

It is sufficient to consider the projection of the trajectory onto the plane $w_3=0$, which is obviously diffeomorphic to the trajectory. Thus, only the first set of conditions needs to be taken into account. One sees immediately that these projections are subsets of elliptic curves or degenerate cubic curves. 
Note that the curves have singular points for $w _2 = 0 $ and  $ w _1 = \frac{ c _2 \pm \sqrt { c _2 ^2 - 3 c ^2 } } { 3c } $.\\
{\em Case 1: $-\sqrt{3} c < c_2< \sqrt{3} c$, $ c_1 \geq 0$}. In this case no stationary points exist, and the set of points satisfying the condition 
is an unbounded connected component of an elliptic curve, and equal to a trajectory.\\
{\em Case 2: $c_2=\pm \sqrt{3}c$, $ c_1>0$}.\\ 
$\bullet$ If $c_1=\pm \frac{16}{9}\sqrt{3} c $ then the curve is degenerate and the set consists of two branches approaching $(w _1 , w _2 )=(\pm \frac{ 1 }{ \sqrt{3} },0 )$ ; each branch is equal to a trajectory .\\
$\bullet$ If $c_1\neq\pm \frac{16}{9}\sqrt{3} c $ the set of points satisfying the condition is an unbounded connected component of an elliptic curve, and equal to a trajectory.\\
{\em Case 3: $-\sqrt{3} c > c_2$ or $c_2> \sqrt{3} c$, $ c_1>0$}. In this case  there are two stationary points corresponding to  the following values of $(w_1,w_2)$: $\left( \frac{ c _2 + \sqrt { c _2 ^2 - 3 c ^2 } } { 3c } ,0\right )$ and $\left( \frac{ c _2 - \sqrt { c _2 ^2 - 3 c ^2 } } { 3c } ,0\right)$. The stationary points are a center and a saddle, respectively. Let $c_1^+$ the value of $c_1$ corresponding to the center, and $c_1^- $ the value of $ c _1 $ corresponding to the saddle, where we have $ c _1^- < c _1^+ $. The sets of points satisfying the condition are depicted in Figure \ref{fig:case3}. \\ 
$\bullet$ If $c_1<c_1^- $  then the set of points satisfying the condition is an unbounded connected component of an elliptic curve, and equal to a trajectory (see the blue curves in Figure \ref{fig:case3}).\\ 
$\bullet$ If $c_1=c_1^- $ the curve is degenerate, and the set consists of two unbounded branches approaching
$\left( \frac{ c _2 - \sqrt { c _2 ^2 - 3 c ^2 } } { 3c } ,0\right)$, and an homoclinic loop. Each of these branches is a trajectory (these are the purple curves in Figure \ref{fig:case3}).\\
$\bullet$ If $c_1^-<c_1<c_1^+ $ then the set is the union of a bounded  and an unbounded connected components of an elliptic curve, each of which is equal to a trajectory (these are the red curves in Figure \ref{fig:case3}). The bounded component is a closed loop.  \\
$\bullet$ If $c_1=c_1^+ $ the curve is degenerate, and the set consists of the union of an unbounded branch and a single point, each of them being a trajectory  (see the black curves in Figure \ref{fig:case3}).\\ 
$\bullet$ If $c_1>c_1^+ $ then the set is an unbounded connected component of an elliptic curve, and equal to a trajectory (see the green curves in Figure \ref{fig:case3}).\\ 
\end{example}

\subsection{Parabolic Rotations}
Finally we discuss the potentials which depend only on $x_2-x_3$, i.e. $U(x)=U(x_2-x_3)$. Such potentials are invariant with respect to the group $G=G_p\cong (\mathbb R,\,+)$ of parabolic rotations

\beq R_p(s)=\begin{pmatrix}
1 & -s & s \\ 
s& 1-s^2/2 & s^2/2 \\ 
s& -s^2/2 & 1+s^2/2 
\end{pmatrix}, 
\eeq
The tangent lift of $R_p(t)$ gives a $G$-action on $T\mathbb H^2$
\[
\Phi_p:G\times T\mathbb H^2\to T\mathbb H^2:(x,y)\to (R_p(t)x, R_p(t)y)
\]
with infinitesimal generator $Y_p$ whose integral curves satisfy

\beq
\begin{split}
 \frac{dx}{dt}&=-x\times_L (e_2+e_3)\\
 \frac{dy}{dt}&=-y\times_L (e_2+e_3)\\
\end{split}
\eeq
Thus $Y_p$ is the Hamiltonian vector field corresponding to the Hamiltonian
\beq
J_p: T\mathbb{R}^{2,1}\to \mathbb{R}:(x,y)\to \langle x\times_Ly,e_2+e_3\rangle=-x_1(y_3-y_2)+y_1(x_3-x_2)
\eeq
which is the momentum mapping of the action $\Phi_h$. One can check that
\[
\{J_p,H\}^*|_{T\mathbb H^2}=0.
\]
that is, $J_p|{T\mathbb H^2}$ is an integral of motion. Therefore the Hamiltonian system that describes the motion of a particle on the hyperbolic plane under the influence of an axisymmetric potential $U(x_2-x_3)$ is Liouville integrable.

The parabolic transformations generate an unipotent one-parameter group, therefore the theory for reductive groups is not applicable.
 However, this group is a maximal parabolic subgroup of the Lorentz group, hence one still has finite generation of the invariants for the action on $\mathbb R^6$; see e.g.  
 Had\v{z}iev \cite{Hadziev}, Vinberg and Popov \cite{Vinberg}, Grosshans \cite{Grosshans83}. Moreover, the {\sc Singular} library  {\tt ainvar.lib}  \cite{PG09} (see also \cite{GPS09}) provides an algorithm for the computation of generators. One obtains (up to some normalizations)

\begin{lemma}\label{pargen}
{\em (a)} The algebra $\mathbb{R}[x,y]^G$ of $G$-invariant polynomials in $(x,y)$ is generated by 
\begin{align*}
\zeta_1 &= x_2-x_3 &\zeta_2 &=y_2-y_3 \\
 \zeta_3 &=y_1^2+y_2^2-y_3^2 & \zeta_4 &=x_1y_1+2x_2y_2-x_2y_3-x_3y_2   \\
\zeta_5 &= x_1^2-2x_2x_3 +2x_2^2    & \zeta_6 &=x_1(y_2-y_3)+y_1(x_3-x_2)
\end{align*}
subject to the relation
\beq
\zeta_1^2\zeta_2^2+\zeta_1^2\zeta_3-2\zeta_1\zeta_2\zeta_4+\zeta_2^2\zeta_5-\zeta_6^2=0.
\label{relations-p}
\eeq
{\em (b)} Restriction from $T\mathbb R^{2,1}$ to $V$ imposes the extra conditions
\beq
\zeta_5-\zeta_1^2=-1,\quad \mbox{ and }\zeta_4-\zeta_1\zeta_2=0,
\label{restriction-p}
\eeq
and rewriting relation (\ref{relations-p}) by eliminating $\zeta_4$ and $\zeta_5$, one finds
\beq
\zeta_1^2\zeta_3-\zeta_2^2-\zeta_6^2=0.
\label{variety-p}
\eeq
Thus the restrictions of the $\zeta_i$ to $V$ (which are denoted by the same symbols) form a subalgebra of $\mathbb R[V]^G$ which is generated by $\zeta_1$, $\zeta_2$, $\zeta_3$ and $\zeta_6$. 
\end{lemma}
We do not (and need not) discuss at this point whether this subalgebra is equal to the invariant algebra for the action of $G$ on the variety $V$. For the purpose of reduction the following observation is sufficient:
\begin{lemma}\label{parexperiment}
 The restrictions of the $\zeta_i$ to the variety $V$ satisfy the identities
\[
\begin{array}{l}
\{\zeta_1,\zeta_2\}^*=\zeta_1^2, \quad \{\zeta_2,\zeta_3\}^*=2\zeta_1\zeta_3, \quad \{\zeta_3,\zeta_1\}^*=-2\zeta_2,\\
\mbox{  and    }\{\zeta_i,\zeta_6\}^*=0,\quad 1\leq i\leq 3.
\end{array}
\]
\end{lemma}

\begin{proposition}\label{parabolicredprop} 
{\em(a)} The orbit of every point on $T\mathbb H^2$ is a parabola which lies in a two-dimensional affine subspace of $\mathbb R^6$.\\
{\em(b)} The map 
\[
\Gamma:\,V\to \mathbb R^4,\quad z\mapsto\left(\begin{array}{c} \zeta_1(z)\\
                                                                                                    \zeta_2(z)\\
                                                                                                    \zeta_3(z)\\
                                                                                                    \zeta_6(z)\end{array}\right),
                                                                                      \]
where we denote with $ w = ( w _1, w _2 , w _3, w _4)^{tr}$ the image of $z$ under $ \Gamma $, sends $V$ to the variety $Y\subseteq \mathbb R^4$ which is defined by 
\begin{equation}\label{parvariety}
w_4^2=w_1^2w_3-w_2^2.
\end{equation}
{\em (c)} The image of $T\mathbb H^2$ under the map $ \Gamma $ is equal to 
\[
\left\{w\in Y;\,w_1<0\right\}.
\]
In particular the Zariski closure of $\Gamma(V)$ equals $Y$.\\
{\em(d)} The Poisson-Dirac bracket on $V$ induces a Poisson bracket $\left\{\cdot,\cdot\right\}^\prime$ on $Y$, which is determined by 
\[
\{w_1,w_2\}^\prime=w_1^2, \quad \{w_2,w_3\}^\prime=2w_1w_3, \quad \{w_3,w_1\}^\prime=-2w_2,
\]
and
\[
\quad \{w_i,w_4\}^\prime=0,\,1\leq i\leq 3.
\]
{\em (e)} The equations of motion \eqref{motion} on $T\mathbb H^2$ for the Hamiltonian 
\[
H=\frac12 \zeta_3 + U(\zeta_1)
\]
with axisymmetric potential are mapped by $\Gamma$ to the reduced Hamiltonian system
\[
\begin{array}{rcccl}
\dot w_1&=&\left\{w_1,h\right\}^\prime&=& w_2\\
\dot w_2&=&\left\{w_2,h\right\}^\prime&=& w_1w_3-w_1^2U^\prime(w_1)\\
\dot w_3&=&\left\{w_3,h\right\}^\prime&=& -2w_2U^\prime(w_1)\\
\dot w_4&=&\left\{w_4,h\right\}^\prime&=& 0
\end{array}
\]
on $Y$, with $h(w):= \frac12 w_3+U(w_1)$. This system admits the first integrals  $h$ and $w_4$. 
\end{proposition}
\begin{proof} Let $(u,v)^{\rm tr}\in\mathbb R^6$. Then the points in its orbit satisfy
\[
\left(\begin{array}{c}
x_1\\
x_2\\
x_3\\
y_1\\
y_2\\
y_3\end{array}\right)=\left(\begin{array}{c}
u_1\\
u_2\\
u_3\\
v_1\\
v_2\\
v_3\end{array}\right)+s\left(\begin{array}{c}
u_3-u_2\\
u_1\\
u_1\\
v_3-v_2\\
v_1\\
v_1\end{array}\right)+\frac{s^2}{2}\left(\begin{array}{c}
0\\
u_3-u_2\\
u_3-u_2\\
0\\
v_3-v_2\\
v_3-v_2\end{array}\right)
\]
On $T\mathbb H^2$ one has $u_1^2+u_2^2-u_3^2=-1$, and therefore $u_3-u_2\not=0$. This shows linear independence of the second and third vector on the right hand side, and part (a) follows.\\
Part (b) is a direct consequence of Lemma \ref{pargen}.\\
To prove part (c), let $w$ be in the image of $\Gamma$, thus there is a point $(x,y)^{\rm tr}\in T\mathbb H^2$ which satisfies
\[
\begin{array}{rcl}
x_2-x_3&=& w_1\\
y_2-y_3&=& w_2\\
y_1^2+2y_2(y_2-y_3)-(y_2-y_3)^2&=& w_3\\
x_1(y_2-y_3)-y_1(x_2-x_3)&=& w_4
\end{array}
\]
with $w_4^2=w_1^2w_3-w_2^3$, and by definition of the variety $V$ one has
\[
\begin{array}{rcl}
x_1^2+\left( 2x_2-(x_2-x_3)\right)(x_2-x_3)&=& -1\\
x_1y_1 +(x_2-x_3)y_2+\left(x_2-(x_2-x_3)\right)(y_2-y_3)&=& 0.
\end{array}
\]
The proof of part (a) shows that every orbit in $T\mathbb H^2$ contains a unique element with $x_1=0$, hence it suffices to seek an inverse image with $x_1=0$. Then the first defining equation for $V$ shows that $w_1\not=0$, and furthermore $x_2=(w_1^2-1)/(2w_1)$. Moreover the first invariant $\zeta_1$ yields
\[
x_3=-\frac{w_1^2+1}{2w_1}
\]
which implies $w_1<0$ in view of $x_3>0$. We proceed to show that every $w\in Y$ with $w_1<0$ admits an inverse image. The second defining identity for $V$ shows
\[
y_2=\frac{1+w_1^2}{2w_1^2}w_2,
\]
and the equality $\zeta_3=w_3$ implies
\[
y_1^2=\frac{1}{w_1^2}\left(w_1^2w_3-w_2^2\right).
\]
This equation for $y_1$ has a real solution in view of \eqref{parvariety}; and an inverse image of $w$ has been found.
\\
The assertion of part (d) follows from Lemma \ref{parexperiment}, and the remainder of the proof is analogous to the previous cases.
\end{proof}
Before proceeding we note and prove one more useful fact.
\begin{lemma}\label{separation} The invariants $\zeta_1$, $\zeta_2$, $\zeta_3$ and $\zeta_6$ separate the orbits of the group action on $T\mathbb H^2$.
\end{lemma}
\begin{proof} We continue the argument from the proof of part (c) of the Proposition. Every orbit in $T\mathbb H^2$ contains exactly one element $(u,v)^{\rm tr}$ with $u_1=0$. Now let $(\hat u,\hat v)\in T\mathbb H^2$ with $\hat u_1=0$ have the same image $w$ under $\Gamma$. Then $u_2-u_3=\hat u_2-\hat u_3=w_1\not=0$, and the first equation defining $V$ shows that
\[
w_1(u_2+u_3)=(u_2-u_3)(u_2+u_3)=-1=(\hat u_2-\hat u_3)(\hat u_2+\hat u_3)=w_1(\hat u_2+\hat u_3)
\]
which implies $u_2=\hat u_2$ and $u_3=\hat u_3$. Thus $u=\hat u$.\\
Furthermore we have $v_2-v_3=w_2=\hat v_2-\hat v_3$, and the second defining equation for $V$ now shows
\[
(u_2-u_3)v_2+u_3w_2=0=(u_2-u_3)\hat v_2+u_3w_2,
\]
hence $v_2=\hat v_2$ and $v_3=\hat v_3$. Finally, the invariant $\zeta_6$ provides
$v_1w_1=\hat v_1w_1$, thus $v=\hat v$.
\end{proof}
As in the elliptic and hyperbolic setting we discuss the map
\[
\widetilde\Gamma=\left(\begin{array}{c}\zeta_1\\
                                              \zeta_2\\
                                            \zeta_3
                                           \end{array}\right),
\]
i.e., the composition of $\Gamma$ and the projection from $\mathbb R^4$ to $\mathbb R^3$. The proof of the next assertion follows from the fact that $\Gamma$ is constant on orbits, from Proposition \ref{parabolicredprop} and Lemma \ref{separation}.

\begin{lemma}\label{parabolicredmap}
{\em (a)} A point $w=(w_1,\,w_2,\,w_3)^{\rm tr}\in\mathbb R^3$ lies in $\widetilde\Gamma(T\mathbb H^2)$ if and only if 
\[
 w_1^2w_3-w_2^2\geq 0 \mbox{  and  } w_1<0.
\]
{\em(b)} The inverse image of $w$ is a union of two parabolas (orbits of the group action) whenever $w_1^2w_3-w_2^2>0$, and a single parabola if $w_1^2w_3-w_2^2=0$.
\end{lemma}

To the map $\widetilde\Gamma$ corresponds a  reduced system determined by the {reduced} Hamiltonian 
\[
 h=\frac 1 2 w_3+U(w_1),
\]
explicitly
\begin{equation}
 \begin{split}
  \dot w_1&=\{w_1, h\}^\prime=w_2\\
\dot w_2&=\{w_2, h\}^\prime=w_1w_3-w_1^2U'(w_1)\\
\dot w_3&=\{w_3, h\}^\prime=-2w_2U'(w_1)
 \end{split}
\label{parabolicredsys}
\end{equation}
with first integrals $h$ and $j^2=w_1^2w_3-w_2^2\geq 0$. Again we discuss the behavior of the reduced system, with rational potential $U$.

\begin{proposition}\label{parabolicredana}
{\em (a)} Relative equilibria: The stationary points of \eqref{parabolicredsys} on the image of  $\widetilde \Gamma$ are
\[
 z_\rho:=\left(\begin{array}{c}\rho\\
                               0\\
                                {\rho}\,U^\prime(\rho)\end{array}\right),\quad\rho<0\mbox{  and  }\rho\,U^\prime(\rho)\geq 0.
\]
The inverse image of $z_\rho$ is a parabola, which is a trajectory of \eqref{motion} whenever $U^\prime(\rho)\not=0$, and consists of stationary points only when $U^\prime(\rho)=0$.\\
{\em(b)} Every nonstationary solution of \eqref{parabolicredsys} is contained in the variety defined by
\[
\begin{array}{rclcl}
j^2&:=&w_1^2w_3-w_2^2&=& c_1\\
h&:=&\frac12 w_3+U(w_1)&=& c_2
\end{array}
\]
with real constants $c_1\geq 0$ and $c_2$. This variety is smooth at every nonstationary point of the reduced system. In particular, if the level set contains no stationary point then it is a one-dimensional submanifold and every connected component is a trajectory of \eqref{parabolicredsys}. If the level set contains a stationary point then the complement of the stationary point set is a union of nonstationary trajectories, with limit sets contained in the set of stationary points.\\
{\em (c)} The inverse image of a nonstationary solution of \eqref{parabolicredsys} on the level set $j^2=c_1$, $h=c_2$ is an unbounded invariant set of \eqref{motion} which contains no stationary points.
\end{proposition}
\begin{remark} One can determine the level sets of $(j^2,h)$ in a direct manner from
\[
2\left(c_2-U(w_1)\right)=w_3=\frac{w_2^2+c_1}{w_1^2}
\]
and thus
\[
w_2^2=2\left(c_2-U(w_1)\right)w_1^2-c_1.
\]

\end{remark}
\begin{example} For linear potential, thus $U(x_2-x_3)= c\cdot (x_2-x_3)$ with $c\not=0$, system  \eqref{parabolicredsys} admits equilibria
\[
z_\rho=\left(\begin{array}{c} \rho\\
                                              0\\
                                          c\rho\end{array}\right), \quad\rho<0,
\]
whenever $c<0$ (see figure \ref{fig:parabolic}). The corresponding solutions of \eqref{motion} can be interpreted as motion on a parabola.  For $c>0$ there exist no stationary points in the image of $\widetilde \Gamma$.\\
We will discuss the trajectories of  \eqref{parabolicredsys} in some detail, starting from the conditions
\[
\begin{array}{c}
w_2^2=2w_1^2\left(c_2-cw_1\right)-c_1,\quad w_1<0\\
     w_3=2\left(c_2-cw_1\right).
\end{array}
\]
It is sufficient to consider the projection of the trajectory onto the plane $w_3=0$, which is obviously diffeomorphic to the trajectory. Thus, only the first set of conditions needs to be taken into account. One sees immediately that these projections are subsets of elliptic curves or degenerate cubic curves. An elementary verification shows that the list is as follows:\\
{\em Case 1: } $c>0$. Then the set of points satisfying the condition is unbounded, in view of $-2c_1^3\to\infty$ as $w_1\to -\infty$. Moreover, no stationary points exist.\\
$\bullet$ If $c_1>0$ then the set is an unbounded connected component of an elliptic curve, and equal to a trajectory.\\
$\bullet$ If $c_1=0$ and $c_2<0$ then the set is an unbounded connected component of an elliptic curve, and equal to a trajectory.\\
$\bullet$ If $c_1=0$ and $c_2>0$ then the curve is degenerate, and the set consists of two branches approaching  $0$ with slopes $\pm \sqrt{2c_2}$; each branch is equal to a trajectory.\\
$\bullet$ If $c_1=0$ and $c_2=0$ then the curve is degenerate, and the set consists of two branches approaching  $0$ with slope $0$; each branch is equal to a trajectory.\\

\begin{figure}
        \centering
        \begin{subfigure}[b]{0.5\textwidth}
                \centering
                \includegraphics[width=\textwidth]{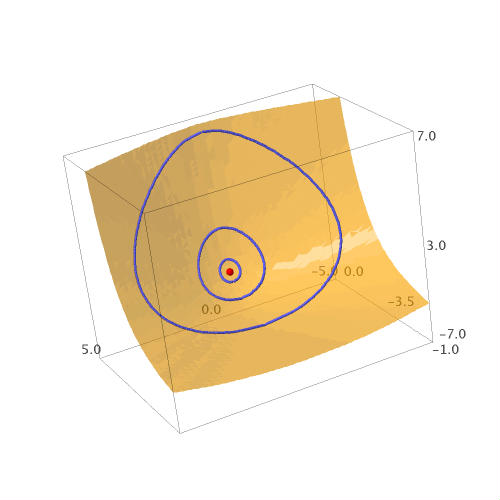}
                \caption{ $ c = -1 $ and $ c _1 = 3 $ }
        \end{subfigure}%
        ~ 
        \begin{subfigure}[b]{0.5\textwidth}
                \centering
                \includegraphics[width=\textwidth]{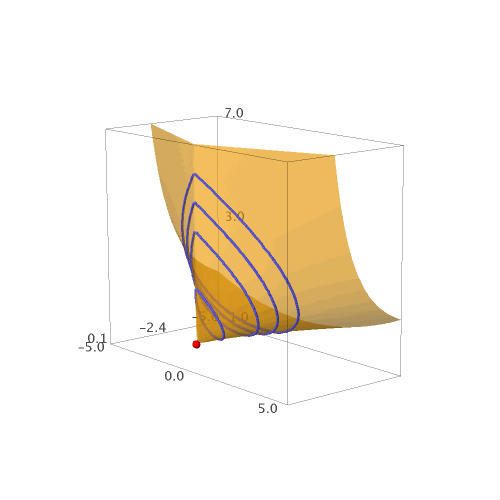}
               \caption{ $ c = -1 $ and $ c _1 = 0 $ }
        \end{subfigure} 
        \caption{Trajectories (in blue) and stationary points (in red) of the reduced system in the parabolic case  with a linear potential.}\label{fig:parabolic}
\end{figure}

{\em Case 2: } $c<0$. Then the set of points satisfying the condition is bounded, in view of $-2c_1^3\to-\infty$ as $w_1\to -\infty$.\\
$\bullet$ If $c_2\leq 0$ then the set is empty.\\
$\bullet$ If $c_2>0$ and $c_1=0$ then the set is a loop of a degenerate curve, which is a trajectory and approaches $0$ with slopes $\pm \sqrt{2c_2}$.\\
$\bullet$ If $c_2>0$ and $c_1>0$ then the set is either a bounded component of an elliptic curve, or a single point (in which case the curve is degenerate and the point is stationary), or empty, depending on the magnitude of $c_1$.
\end{example}

\appendix
 \section{ Poisson varieties and reduction} 
In this section we discuss the notion of a Poisson variety, which is a natural generalization of Poisson manifold. Initially we will restrict the class of admissible functions to polynomials (or rationals). Within this framework we obtain a singular reduction procedure for actions of linear algebraic groups. Such actions occur naturally in a number of relevant applications, including those discussed in the main part of the present paper. Our approach may be seen as complementary to the singular reduction procedure developed in Cushman and Bates \cite{Cushman}, Appendix B. A priori the settings, and to some extent the objectives, of both procedures are different. Cushman and Bates consider a Lie group that acts smoothly on a symplectic manifold and respects the symplectic structure. Their construction starts from topological properties of the quotient of this action (the orbit space) and in order to ensure good topological behavior, properness of the group action is required in \cite{Cushman}. On the other hand, we start with different assumptions on the group and the group action and use some facts and tools from Algebraic Geometry. From a technical perspective the approach via Poisson varieties is quite uncomplicated, and the same holds for computational issues, provided that the invariant algebra admits a set of generators that is not only finite but also small enough to keep computations feasible. The ``quotient" obtained in our procedure is a subset of an algebraic variety. There is generally no $1$-$1$-correspondence between this subset and the orbits of the group action, but from an algebraic point of view the variety may still be seen as a good substitute for a quotient, and one thus obtains a natural ``reduced Poisson bracket". 
(A similar approach was used ad hoc in \cite{WalcherNF} to discuss certain Hamiltonian Poincar\'e-Dulac normal forms.) For linear compact groups acting naturally on $n$-space (or on a non-singular variety), both notions - by virtue of a theorem by Schwarz \cite{Schwarz} - provide the same reduction for smooth functions.

\medskip
\noindent
We will present  a rather detailed outline of differential equations on varieties, as well as their reduction with respect to an algebraic group action, including some proofs, for the convenience of the reader.  To the best of our knowledge, Proposition \ref{poiredprop} has not appeared in the literature so far.

We first recall a few facts about algebraic varieties and introduce some notation. The necessary background material can be found in concise form e.g. in the first chapter of the monograph \cite{Kunz} by Kunz, or in the introductory Chapter I of Humphreys' book on algebraic groups \cite{Humphreys}.

We let $\mathbb K$ stand for $\mathbb R$ or $\mathbb C$. 
Consider the affine space $\mathbb K^n$ and the algebra  $\mathbb K[x_1,\ldots, x_n]$ of polynomial functions on $\mathbb K^n$. By Hilbert's Basissatz any ideal in $\mathbb K[x_1,\ldots, x_n]$ admits a finite set of generators.
We call $V\subset\mathbb K^n$ Zariski-closed (or an affine algebraic variety, briefly a variety) if it is the common zero set of some family of polynomials. The set of all polynomials vanishing on $V$ is an ideal of $\mathbb K[x_1,\ldots, x_n]$, which is called the vanishing ideal $J(V)$.  The Zariski-closed sets of $\mathbb K^n$ define the Zariski topology on $\mathbb K^n$ and, by restriction, on its subsets. The vanishing ideal of a variety is always a radical ideal (i.e., whenever $\psi^m\in J(V)$ for some $m>0$ then $\psi\in J(V)$);  there is a $1$-$1$-correspondence between varieties and radical ideals. A variety $V$ is called irreducible if it is not the union of two proper closed subsets; this holds if and only if its vanishing ideal is prime.
A regular function on the variety $V$ is - by definition - the restriction of a polynomial in $\mathbb K[x_1,\ldots, x_n]$ to $V$; the algebra  $\mathbb{K}[V]$ of regular functions is therefore isomorphic to the quotient algebra $\mathbb K\left[x_1,\ldots,x_n\right]/J(V)$. (For the sake of clarity we will sometimes denote a regular function as a class $\phi+J(V)$.)

\medskip
\noindent
We now turn to differential equations and vector fields on varieties. Given a vector field $f$ with components $f_1,\ldots, f_n\in \mathbb K[x_1,\ldots, x_n]$, consider the 
polynomial ordinary differential equation
\[
\dot x=f(x)=\left(\begin{array}{c}f_1(x)\\
                                                 \vdots\\
                                                  f_n(x)\end{array}\right)
                                                   \mbox{  on  }\mathbb K^n.
\]
The local flow of this equation, i.e. the solution with initial value $y$ at $t=0$, will be denoted $\Phi_f(t,\,y)$.
To this differential equation corresponds a derivation of $\mathbb K[x_1,\ldots, x_n]$, given by the Lie derivative $L_f$ with
 $L_f(\phi)(x)=D\phi(x)\,f(x)$ for any function $\phi$. Moreover, every derivation of  $\mathbb K[x_1,\ldots, x_n]$ is of this type.

We recall a standard result on invariant algebraic varieties.
\begin{lemma}\label{appleminv}
The local flow of the polynomial differential equation $\dot x=f(x)$ on $\mathbb K^n$ leaves $V$ invariant if and only if $L_f$ sends $J(V)$ to $J(V)$.
\end{lemma}
\begin{proof}
Let $\psi$ be any polynomial. Then one has the Lie series identity
\[
\psi\left(\Phi_f(t,y)\right)=\sum_{k=0}^\infty\frac{t^k}{k!}L_f^k(\psi)(y)
\]
(see e.g. Gr\"obner/Knapp \cite{Groebner}). Assuming that $L_f$ sends $J(V)$ to itself, one obtains for $\psi\in J(V)$ that all $L_f^k(\psi)\in J(V)$, and thus all $L_f^k(\psi)(y)=0$ whenever $y\in V$. Conversely, assume that $V$ is an invariant set for the differential equation and let $y\in V$, $\psi\in J(V)$. Then $\psi\left(\Phi_f(t,y)\right)=0$ for all $t$ near $0$, which in particular implies $L_f(\psi)(y)=0$.
\end{proof}
 If the condition Lemma \ref{appleminv} holds then $L_f$ induces a derivation of $\mathbb K[V]/J(V)$. Thus one has a correspondence between these derivations and polynomial differential equations with invariant set $V$.

\medskip
\noindent
\begin{definition}
A {\em Poisson bracket} on the variety $V$ is a skew-symmetric bilinear map 
\[
\left\{\cdot,\,\cdot\right\}:\,\, \mathbb{K}[V]\times \mathbb{K}[V]\to \mathbb{K}[V]
\]
which satisfies the Leibniz rule and the Jacobi identity.
\end{definition}
A Poisson bracket on a variety may be characterized by polynomial ``structure functions" as follows:
For $1\leq i,\,j\leq n$ let $\gamma_{ij}\in \mathbb K\left[x_1,\ldots,x_n\right]$ such that 
\[
\left\{x_i+J(V),\,x_j+J(V)\right\}=\gamma_{ij}+J(V)
\]
for all $i$ and $j$. Then, by bilinearity and the Leibniz rule, for all polynomials $\phi$ and $\psi$ one has
\begin{equation}\label{poisdef}
\left\{\phi+J(V),\,\psi+J(V)\right\}=\sum_{i,j}\gamma_{ij}\frac{\partial \phi}{\partial x_i}\cdot\frac{\partial \psi}{\partial x_j}+J(V)
\end{equation}
By construction the polynomial map $\left(\sum_{i}\gamma_{ij}\frac{\partial \phi}{\partial x_i}\right)_{1\leq j\leq n}$ gives rise to a vector field which sends $J(V)$ to itself; this will be called the Hamiltonian vector field of $\phi+J(V)$.

\medskip
\noindent
Next we turn to symmetry groups. Let $G\subseteq GL(n,\,\mathbb K)$ be a linear algebraic group acting naturally on $\mathbb K^n$, and let $V$ be an irreducible variety such that $T(V)=V$ for all $T\in G$ (thus $G$ is a subgroup of the automorphism group of $V$). We will assume, in addition, that the invariant algebra $\mathbb K[V]^G$ is finitely generated, and that $\gamma_1,\ldots,\gamma_r$ is a set of generators. Let 
\[
I:=\left\{\psi\in\mathbb K[x_1,\ldots,x_r]:\, \psi(\gamma_1,\ldots,\gamma_r)=0\right\}
\]
denote the ideal of relations between the $\gamma_i$. This is obviously a radical ideal. We define the Hilbert map
\[
\Gamma=\left(\begin{array}{c}\gamma_1\\
                                              \vdots\\
                                            \gamma_r\end{array}\right): V\to \mathbb K^r.
\]
\begin{lemma}\label{applemhilb} The Zariski closure of $\Gamma(V)$ is equal to the vanishing set $Y$ of $I$, and $I=J(Y)$.
\end{lemma}
\begin{proof}
One has
\[
\begin{array}{rcl}
\overline{\Gamma(V)}&=& \left\{ z\in \mathbb K^r:\, \rho(z)=0\mbox{  for all  }\rho\mbox{  with  }\rho\circ\Gamma=0\right\}\\
   &=&  \left\{ z\in\mathbb K^r:\, \rho(z)=0\mbox{  for all  }\rho\in I \right\}.
\end{array}
\]
This implies both assertions.
\end{proof}

Now assume furthermore that the polynomial vector field $f$ on $V$ is $G$-symmetric, thus
$Tf(y)=f(Ty)$ for all $T\in G$ and $y\in V$. (Note that $f$ is determined only up to elements of $J(V)^n$, and that the symmetry condition must be satisfied modulo $J(V)^n$ only.) Then the Hilbert map induces a reduced vector field on $Y\subseteq\mathbb K^r$.
\begin{proposition}\label{redprop} There is a polynomial vector field $g$ on $\mathbb K^r$ such that the identity
\[
D\Gamma(x)f(x)=g(\Gamma(x))
\]
holds on $V$. Thus $\Gamma$ maps parameterized solutions of $\dot x=f(x)$ on $V$ to parameterized solutions of $\dot y=g(y)$. Moreover, the variety $Y$ is invariant for $\dot y=g(y)$.
\end{proposition}
\begin{proof} Since $f$ is $G$-symmetric, one has $L_f(\psi)\in \mathbb K[V]$ for every $\psi\in\mathbb K[V]$. In particular there exist polynomials $g_i$ in $r$ variables such that 
\[
L_f(\gamma_i)=g_i(\gamma_1,\ldots,\gamma_r),\quad 1\leq i\leq r,
\]
hence there exists $g$ such that the asserted identity holds. This, in turn, implies the identity
\[
\Gamma(\Phi_f(t,z))=\Phi_g(t,\Gamma(z)),\quad z\in V,
\]
for sufficiently small $t$. In particular $\Phi_g(t,\Gamma(z))\in \Gamma(V)$ for all $t$ near $0$. Let $\rho$ be an element of $I$, the ideal of relations between the $\gamma_i$. By the Lie series identity one obtains $L_g(\rho)(w)=0$ for all $w\in \Gamma(V)$. But then
$L_g(\rho)(w)=0$ for all $w\in \overline{\Gamma(V)}=Y$, and the second assertion follows with Lemma \ref{appleminv}.
\end{proof}

After these preliminaries, we turn to symmetry reduction of Poisson varieties. Thus, let $V$ and $G$ be as above, and moreover assume that there is a Poisson bracket on $V$ which is  $G$-invariant, i.e., the identity
\[
\left\{\phi\circ T,\,\psi\circ T\right\}=\left\{\phi,\,\psi\right\}\circ T
\]
holds in $\mathbb K[V]$ for all $T\in G$. This condition directly implies that the Poisson bracket of two $G$-invariant regular functions on $V$ is $G$-invariant.

\begin{proposition}\label{poiredprop}
{\em (a)} Let $\phi,\,\psi \in \mathbb K[V]^G$, and let $\widetilde\phi$, $\widetilde\psi\in\mathbb K[x_1,\ldots,x_r]/I$ such that
\[
\phi=\widetilde\phi\circ\Gamma,\quad \psi=\widetilde\psi\circ\Gamma.
\]
Then there exists a uniquely determined $\{\widetilde \phi,\,\widetilde\psi\}^\prime\in\mathbb K[x_1,\ldots,x_r]/I=\mathbb K[Y]$ such that
\begin{equation}\label{poisred}
\{\phi,\,\psi  \}  =\{\widetilde \phi,\,\widetilde\psi\}^\prime\circ\Gamma.
\end{equation}

\noindent
{\em (b)} Identity \eqref{poisred} defines a Poisson bracket $\{\cdot,\,\cdot\}^\prime$ on $Y$. The Hilbert map $\Gamma$ sends Hamiltonian $G$-equivariant vector fields on $V$ to vector fields on $Y$ that are Hamiltonian with respect to  $\{\cdot,\,\cdot\}^\prime$ .
Moreover, Casimir elements of $(V,\,\left\{\cdot,\,\cdot\right\})$ are sent to Casimir elements of $(Y,\,\left\{\cdot,\,\cdot\right\}^\prime)$.
\end{proposition}
\begin{proof}
Part (a) holds because $\{\phi,\,\psi  \}$ is $G$-invariant, and the image of $\Gamma$ is Zariski-dense in $Y$.
Part (b) follows in a straightforward manner from the density of $\Gamma(V)$ in $Y$, and from the Poisson bracket properties of $\{\cdot,\,\cdot\}$, using identity \eqref{poisred}.
As a representative we prove the Leibniz rule. Given 
\[
\widetilde\phi,\,\widetilde\psi,\,\widetilde\eta\in\mathbb K[Y]=\mathbb K[x_1,\ldots,x_r]/I,
\]
one has
\[
\begin{array}{lcl}
\{\widetilde\phi,\,\widetilde\psi\cdot\widetilde\eta\}^\prime\circ\Gamma&=&\{\widetilde\phi\circ\Gamma,\,\widetilde{(\psi\cdot\eta)}\circ\Gamma\}=\{\widetilde\phi\circ\Gamma,\,(\widetilde\psi\circ\Gamma)\cdot(\widetilde\eta\circ\Gamma)\}\\
&=&\{\widetilde\phi\circ\Gamma,\,\widetilde\psi\circ\Gamma\}\cdot(\widetilde\eta\circ\Gamma)+\{\widetilde\phi\circ\Gamma,\,\widetilde\eta\circ\Gamma\}\cdot(\widetilde\psi\circ\Gamma)\}\\
&=&\left (\{\widetilde\phi,\,\widetilde\psi\}^\prime\cdot\widetilde\eta+\{\widetilde\phi,\,\widetilde\eta\}^\prime\cdot\widetilde\psi\right)\circ\Gamma.
\end{array}
\]
by the Leibniz rule in $\mathbb K[V]$ and various definitions.
This implies
 \[
 \{\widetilde\phi,\,\widetilde\psi\cdot\widetilde\eta\}^\prime\\
=
\{\widetilde\phi,\,\widetilde\psi\}^\prime\cdot\widetilde\eta+\{\widetilde\phi,\,\widetilde\eta\}^\prime\cdot\widetilde\psi
\]
by Zariski density of the image of $\Gamma$. All the remaining properties, in particular the Jacobi identity, are proven in the same fashion.
\end{proof}
In applications one may use identity \eqref{poisred} to compute the ``structure functions" of the Poisson bracket on $Y$ and then employ identity \eqref{poisdef}, mutatis mutandis. We note that Proposition \ref{poiredprop} provides a natural construction of Poisson varieties if one starts with a Hamiltonian vector field on $\mathbb R^{2n}$ (usual symplectic structure), and a linear (algebraic) group $G$ which leaves the standard Poisson bracket invariant: Reduction by invariants of $G$ (assuming finite generation of the invariant algebra) will lead to a Poisson variety. However, Poisson varieties occur in different ways, such as the smooth Poisson variety $V\supseteq\mathbb H^2$ which is in the focus of the present paper.

\medskip
\noindent
Finally, we note some observations about extending the class of admissible functions.
\begin{remark}
If one starts with a (local) analytic function or a $C^\infty$ function
$\widetilde \mu$ in $r$ variables, defined in a neighborhood of some $y\in Y$, then one may consider $\mu:=\widetilde\mu\circ\Gamma$ as a (local) analytic or $C^\infty$ function on $V$ (this is consistent with \cite{Cushman}, Appendix B.5), and the Hamiltonian vector field of $\widetilde \mu$ on $Y$ (defined via extension of the Poisson bracket \eqref{poisdef}) is the reduced vector field of the Hamiltonian of $\mu$ on $V$.  In this sense, the reduction procedure works for more general functions and vector fields.\\
The critical question, however, is whether every $G$-invariant function of some differentiability class on $V$ can be expressed as a function (of the same class) in the $\gamma_i$.  For compact groups $G$ and $C^\infty$ functions this holds due to a theorem by Schwarz \cite{Schwarz}, and
furthermore one can verify that the reduction we have obtained is equivalent to the singular reduction described in Cushman and Bates \cite{Cushman}. (Essentially this follows from the identification of the orbit space with a semi-algebraic subvariety of $Y$.)
On the other hand, there exist group actions with smooth invariants that cannot be expressed as smooth functions composed with invariant polynomials. Bates \cite{Bates} provides an example from mechanics, with a unipotent symmetry group.
For reductive groups every analytic $G$-invariant function can be written as an analytic function in the $\gamma_i$, as was shown by  Luna \cite{Luna}. 
\end{remark}
\proof[Acknowledgements]

M.S. was partially supported by an NSERC Discover grant, by the DFG-Graduiertenkolleg ``Experimentelle und konstruktive Algebra" during a visit to RWTH Aachen, as well as by the Department of Mathematics during a visit to TU M\"unchen.  J.S. and S.W. gratefully acknowledge the support of the Research in Pairs program of Mathematisches Forschungsinstitut Oberwolfach (MFO) in 2012.

\end{document}